\documentclass[12pt]{amsart}
\usepackage{amsthm}
\usepackage{amsfonts}
\usepackage{comment}
\usepackage{amssymb}
\usepackage{comment}
\usepackage{hyperref}
\usepackage{amsmath}
\usepackage{mathtools}
\usepackage{tikz-cd}
\usepackage[margin=1in]{geometry}

\numberwithin{equation}{section}
\usepackage{breqn}
\usepackage{enumitem}
\usepackage{setspace}

\newtheorem{lemma}{Lemma}[section]
\newtheorem{theorem}[lemma]{Theorem}
\newtheorem{prop}[lemma]{Proposition}

\newtheorem{claim*}{Claim}

\theoremstyle{definition}
\newtheorem{remark}[lemma]{Remark}

\newtheorem{defn}[lemma]{Definition}

\usepackage{parskip}
\usepackage{bm}

\newcommand{\F}{\mathbb{F}}

\newcommand{\A}{\mathbb{A}}
\newcommand{\calC}{\mathcal{C}}
\newcommand{\calS}{\mathcal{S}}
\newcommand{\calZ}{\mathcal{Z}}

\newcommand{\Fph}{\F_{p^h}}
\newcommand{\Zgamma}{\calZ_\gamma}

\DeclareMathOperator{\Tr}{Tr}
\DeclareMathOperator{\N}{N}

\DeclareMathOperator{\supp}{Supp}
\usepackage{xcolor}

\title{Codes with Hierarchical Locality on Artin-Schreier Surfaces}
\author{Jennifer Berg}
\address{Jennifer Berg, Bucknell University, Department of Mathematics, Lewisburg, PA 17837, USA}
\email{jsb047@bucknell.edu}
\author{Beth Malmskog}
\address{Beth Malmskog, Colorado College, Department of Mathematics and Computer Science, Colorado Springs, CO 80903, USA}
\email{bmalmskog@coloradocollege.edu}
\author{Mckenzie West}
\address{Mckenzie West, University of Wisconsin-Eau Claire, Department of Mathematics, Eau Claire, WI 54701, USA}
\email{westmr@uwec.edu} %mckenzierwest@gmail.com

\begin{document}
\maketitle
\begin{abstract}
 In this article, we construct codes with hierarchical locality using natural geometric structures in Artin-Schreier surfaces of the form $y^p-y=f(x,z)$.  Our main theorem describes the codes, their hierarchical structure and recovery algorithms, and gives parameters.  We also develop a family of examples using codes defined over $\F_{p^2}$ on the surface $y^p-y=x^{p+1}z^2+x^2z^{p+1}$.  We use elementary methods to count the $\F_{p^2}$-rational points on the surface, enabling us to provide explicit hierarchical parameters and a better bound on minimum distance for these codes.  An additional example and some generalizations are also considered.
\end{abstract}

\section{Introduction}\label{sec:Introduction}

Error-correcting codes are mathematical objects that encode information with redundancy so that the correct original data is available even when errors or erasures occur.  Encoded data is stored or transmitted in vectors, known as codewords.  In a locally recoverable code, the information in each position of any codeword can be recovered by accessing a relatively small subset of other positions, known as the recovery set for that position.  This property is particularly desirable for cloud storage applications, where a single codeword would be spread out across many servers, with each server storing a single position.  If a server crashes or becomes unavailable because of high demand, a locally recoverable code allows the information on that server to be reconstructed from a small set of other servers.  Since multiple servers could become unavailable, possibly compromising a server's recovery set, it is desirable to have an efficient back-up plan.  Of course, one option is to simply use the whole codeword to recover the missing data by traditional error correction.  However, this would require accessing all the other symbols in the codeword, which would potentially mean querying a very large number of servers.  Two other ideas to address this issue more efficiently are availability and hierarchy \cite{barg2017locally, sasidharan2015codes}.  A code is said to have availability $t$ if each position has $t$ independent recovery sets.  Thus if one recovery set is unavailable, another can be used.  A code is said to have hierarchical locality if each position has at least two nested recovery sets, with the larger recovery set of larger minimum distance, enabling recovery of more erasures than the smaller recovery set.  Thus, if some positions in the smaller set are not available, the larger set can recover all of the missing data.  This paper focuses on codes with hierarchical locality.

Algebraic geometric code constructions have been rich sources of locally recoverable codes, and further of those with availability and hierarchy \cite{barg2017locally, ballentine2019codes, barg2017locallyrecoverable, lopez2021hermitian, salgado2021, haymaker2023algebraic}. Underlying algebraic and geometric structures of algebraic varieties can give rise to useful structures in evaluation codes constructed on those varieties.  In this paper, we construct codes with hierarchical locality by exploiting the algebraic geometric structure of \emph{Artin-Schreier surfaces}.  Such a surface is defined over a field of prime characteristic $p$, by an affine equation of the form $y^p-y=f(x,z)$ for $f(x,z)$ a rational function.  The reason for focus on these surfaces is that, when intersected with planes of constant $z$, the result is often an Artin-Schreier curve of the form $y^p-y=\tilde{f}(x)$ for a non-constant function $\tilde{f}$.  While Artin-Schreier curves have been widely studied for both theoretical interest and applications to coding theory, less is known about the higher-dimensional analogues.  We define evaluation codes on Artin-Schreier surfaces with small recovery sets arising from fibers of constant $x$ and $z$, while larger recovery sets arise from points on Artin-Schreier curves at the intersection of the surface with planes of constant $z$.  A main feature of interest in our code construction is its novel use of geometry and fibrations of surfaces into curves to provide locality and hierarchy for codes. This is a proof of concept for an approach that can potentially be applied to general surfaces, creating a vast array of examples of codes with hierarchical recovery.  A second feature of interest in this work is in the point counting for the Artin-Schreier surfaces in our examples.  Understanding the number and structure of points on these surfaces is of independent value, and contributes to a relatively small body of knowledge on point counts of such surfaces.  

The structure of this paper is as follows.  In Section \ref{sec:Background}, we briefly describe the objects of our study and some useful known results.  The main theorem of this paper is Theorem \ref{thm: main}, where we define evaluation codes on affine Artin-Schreier surfaces and, given some conditions on parameter choices and point counts on curves contained in the surfaces, we show that these codes will have hierarchical locality and determine hierarchical parameters of the codes. Local recovery using the smaller recovery set is by straightforward single-variable polynomial interpolation. Recovery by the larger set requires a specialized multivariate polynomial interpolation algorithm, which we describe in the proof of Theorem \ref{thm: main}. This theorem and the proof are the subject of Section \ref{sec:main theorem}.  Because this theorem is stated and proved in generality, it is not possible to give a non-trivial bound on the minimum distance of the codes at this level. However, in Section \ref{sec:Example}, we develop an example of this construction in detail and determine a more meaningful minimum distance bound.  A related example is described in Section \ref{sec:RelatedExample}.  In Section \ref{sec:Conclusion}, we briefly discuss potential next directions for investigation.

\subsection{Dedication}

This paper is dedicated with respect and affection to Professor Sudhir Ghorpade.  His work weaves together the beauty of algebra, geometry, number theory, and coding theory.  Our work in this paper relies on a perspective that he helped to develop, and that he continues to expand.  We also want to pay tribute to Professor Ghorpade's joyful dedication to teaching, mentoring young researchers, and writing about mathematics.  His warm and inclusive way of interacting with others has made the mathematical community better for all of us.  Thank you, Sudhir!

\subsection{Acknowledgements}

This work began thanks to travel supported by the University of Wisconsin-Eau Claire Department of Mathematics. 
Author Malmskog was partially supported by the National Science Foundation through award number 2137661.
Author West was partially supported by the University of Wisconsin-Eau Claire Vicki Lord Larson and James Larson Tenure-track Time Reassignment Collaborative Research Program.

\section{Background}\label{sec:Background}

\subsection{Notation} Let $\mathbb{Z}^+$ denote the set of positive integers.  Let $q$ be a prime power, and denote the field with $q$ elements by $\F_q$.  For $n\in\mathbb{Z}^+$, let $[n]\coloneqq\{1,2,\dots,n\}$.  We will often consider curves and surfaces embedded in $3$-space.  We use $\mathbb{A}^3$ to denote affine $3$-space over a field $K$ that will be clear from context, i.e. $\mathbb{A}^3=\{(x,y,z):x,y,z \in K\}$.  Throughout this paper, we will refer to projection maps $\pi$ with subscripts indicating the coordinate onto which we are projecting.
For example, for $(a,b,c)\in\mathbb{A}^3$, define
\begin{equation}
    \pi_x\colon \mathbb{A}^3\to \mathbb{A}^1\quad \text{by}\quad \pi_x(a,b,c)=a,
\end{equation}
and 
\begin{equation}
    \pi_{x,z}\colon \mathbb{A}^3\to \mathbb{A}^2\quad \text{by}\quad \pi_x(a,b,c)=(a,c).
\end{equation} 
%Generally speaking, $x$ is the first coordinate, $y$ is the second, and $z$ is the third, though we will rely on context to make this clear.

\subsection{Codes}\label{subsec:BackgroundEvalCodes}

Let $n,k,d\in\mathbb{Z}^+$. A \emph{linear code} $C$ of \emph{length} $n$ and \emph{dimension} $k$ is a $k$-dimensional subspace of $\F_q^n$.  
Elements of $C$ are called \emph{codewords}, and the distance between two codewords is the number of positions in which they differ.
The minimum distance of $C$ is the minimal distance between any two codewords.  We refer to a code with these parameters as an $[n,k,d]$-code.  Given an $[n,k,d]$-code $C$, the \emph{punctured code} to positions in $I=\{i_1,i_2,\dots, i_s\}\subset[n]$ is $C|_I\coloneqq\{(c_{i_1}, c_{i_2},\dots c_{i_s}):(c_1,c_2,\dots,c_n)\in C\}$.  Equivalently, the punctured code can be thought of as the image of $C$ under the projection of $\A^n$ onto $\A^{s}$, where we define
\[\pi_I\colon\A^n\to \A^{s} \quad\text{by}\quad \pi_I((c_1,c_2,\dots,c_n))=(c_{i_1}, c_{i_2},\dots c_{i_s}).\]

Given an ordered finite set $A=\{a_1,a_2,\dots a_n\}$ of cardinality $n$, and $V$ a set of functions from $A$ to $\F_q$, we define the evaluation on $A$ map by \[ev_A:V\rightarrow \F_q^n, \quad ev_A(f)=\left(f(a_1), f(a_2),\dots, f(a_n)\right).\]  If $V$ is an $\F_q$-vector space, then so is $ev_A(V)\coloneqq \{ev_A(f):f\in V\}$, thus $ev_A(V)$ is a linear code. Note that in this case, if $ev_A$ is injective on $V$, then the dimension of $ev_A(V)$ equals $\dim(V)$.

Evaluation codes using sets of points on an algebraic variety and functions defined on these points are often called algebraic geometry (AG) codes.  Generalizations of Reed-Solomon codes, the first AG codes on curves were devised by Goppa in 1982.  AG codes on higher dimensional varieties followed, generalizing $q$-ary Reed-Muller codes. In our setting, let $\calS$ be a surface defined over a finite field $\mathbb{F}_q$ and let $T\subseteq\calS(\F_q)$.  If $\#T=n$, then we arbitrarily order $T$ as $T=\{P_1, P_2,\dots, P_n\}$. 
    Let $V$ be a linear subspace of the function field $\F_q(\calS)$ such that no function in $V$ has poles at any point in $T$. 
    The evaluation code $C(V,T)$ is the vector space
    \[ C(V,T)\coloneqq ev_T(V)=\{ev_T(f):f\in V\}.\]

Codes on algebraic surfaces were introduced by Aubry in his 1993 PhD thesis \cite{AubryThesis}. Our evaluation codes are in the tradition of this work, which focused on the restriction of projective Reed-Muller codes to the points of algebraic surfaces.  In this paper, we restrict strategically-defined subcodes of affine Reed-Muller codes to the affine points of algebraic surfaces.  

Interest in the arithmetic geometry of surfaces and related properties of AG codes on surfaces has grown, leading to exciting progress in the area.  A major part of finding the parameters of such codes is counting the points on the surfaces and on related curves.  In a nice example of this, Couvreur defined good codes on rational surfaces obtained by blowing up the projective plane at a small number of closed points, with the estimate of minimum distance becoming equivalent to counting points on related plane curves \cite{couvreur2011}.  While a survey of all the good work in this area is beyond the scope of this paper, we note some interesting recent results include lower bounds for minimum distance of particular families of surface codes \cite{Aubry2021}, codes on Hirzebruch surfaces \cite{Nardi2019}, good codes from del Pezzo surfaces \cite{blache2020anticanonical}, codes on abelian surfaces \cite{aubry2021algebraic}, and a method for bounding minimum distance on codes from algebraic surfaces with connections to codes on towers of surfaces \cite{couvreur2020toward}.

\subsection{Local and Hierarchical Recovery}\label{subsec:BackgroundHierarchy}

As mentioned in Section \ref{sec:Introduction}, applications such as cloud storage can make it desirable for codes to have additional structure.  Locally recoverable codes (LRCs) were introduced in 2012 by Gopalan, Huang, Semitci, and Yekhanin \cite{Gopalan2012}.  The idea was to design codes so that a single erased symbol could be recovered using only a small number of other symbols, instead of requiring access to at least $k$ codeword symbols (as might be necessary for standard error correction).  LRCs have been studied by many researchers since and were generalized to allow the recovery of more than one erasure, as in the following definition.

\begin{defn} \cite{Prakash}\label{LRCdef2} Let $n, r, \rho \in \mathbb{Z}^+$. A linear code $C$ of length $n$ over $\mathbb{F}_q$ is \emph{$(r,\rho)$-locally recoverable} if for each $1\leq i\leq n$ there exists a punctured code $C_i$ of length  $r+\rho-1$ and minimum distance at least $\rho$ such that the support of $C_i$ contains $i$. The support of $C_i$ is called a \emph{repair group} for position $i$, and $I_i \coloneqq \supp C_i \setminus \{i\}$ is called a \emph{recovery set} for $i$.
\end{defn}

Locally recoverable codes from covering maps of curves were introduced in \cite{barg2017locally}, and further locally recoverable codes from curves and surfaces were described in \cite{barg2017locallyrecoverable}.  Further interesting results on locally recoverable codes on surfaces appear in \cite{salgado2021}. 

An $(r,\rho)$-LRC has the property that any $\rho-1$ erasures within a repair group can be recovered by accessing the remaining symbols of the group; thus the single erasure version corresponds to $\rho=2$.  However, if there are more than $\rho-1$ erasures, or if more than $\rho-1$ symbols are unavailable due to high demand, it is desirable to have a back-up plan to recover these symbols without accessing all remaining symbols of the codeword.  One version of this back-up plan, advanced by Sasidharan, Agarwal, and Kumar in 2015, involves designing codes where the small repair group is nested inside a larger repair group with a larger minimum distance. 

\begin{defn} [\cite{sasidharan2015codes}]\label{HLRCdef}
Let $n, n_1, n_2, s_1, s_2, d_1, d_2\in\mathbb{Z}^+$ with $n_2<n_1$, $s_2\leq s_1$, and  $d_2<d_1$.  A linear code $C$ of length $n$ is said to have \textit{hierarchical locality} with parameters $((n_1, s_1, d_1), (n_2,  s_2, d_2))$ if for each $1\leq i\leq n$ there exists a punctured code $C_{1,i}$ of length $n_1$, dim$(C_{1,i})\leq s_1$, and minimum distance $d(C_{1,i})\geq d_1$ with $i$ in the support of $C_{1,i}$, and such that $C_{1,i}$ is an $ (s_2,d_2)$-locally recoverable code of length $n_2$.  Let $C_{2,i}$ be the local repair group of $C_{1,i}$ with support containing position $i$.
\end{defn}

We call the code $C_{1,i}$ the \emph{middle code} and $C_{2,i}$ the \emph{lower code} for position $i$.  Note that our notation varies slightly from \cite{sasidharan2015codes}, because we would like to be able to refer explicitly to the lower code in our construction.  Codes with hierarchical locality are referred to as HLRCs.  HLRCs from algebraic geometry constructions were considered by Ballentine, Barg, and Vl{\u{a}}du{\c{t}} \cite{ballentine2019codes} using covering maps, quotients, and fiber products of curves. Recent work of Haymaker, Malmskog, and Matthews describes hierarchical locality for Reed-Muller codes as well as fiber product and curve cover constructions \cite{haymaker2023algebraic}. In \cite{haymaker2023algebraic}, the affine subspaces are used to define middle codes for hierarchical recovery of Reed-Muller codes.  This work builds on \cite{haymaker2023algebraic} in the sense that we use the intersection of affine subspaces with a surface in 3-space to define middle codes for hierarchical recovery.  

\subsection{Arithmetic in Characteristic $p$ and Artin-Schreier Covers}\label{subsec:BackgroundAS}

Let $h\in\mathbb{Z}^+$.  The field trace map $\Tr_{\F_{q^h}/\F_{q}}:\F_{q^h}\rightarrow\F_{q}$ is defined by\[\Tr_{\F_{q^h}/\F_{q}}(x)\coloneqq \sum_{i=0}^{h-1} x^{q^i}.\]   This map is a surjective homomorphism of the additive groups of these fields and is thus a $q^{h-1}$-to-one map.  The field norm map $\N_{\F_{q^h}/\F_{q}} \colon \F_{q^h}\rightarrow\F_{q}$ is defined by \[\N_{\F_{q^h}/\F_{q}}(x)\coloneqq \prod_{i=0}^{h-1} x^{q^i}.\] The norm map is a surjective homomorphism of the multiplicative groups of the fields, and is $(q^{h-1}+1)$-to-one map except at $0$, where it is one-to-one. Note that field norm and trace maps exist for all finite-degree field extensions, not only in characteristic $p$.

An Artin-Schreier extension is any degree-$p$ Galois extension of a field of characteristic $p$.  A good reference on basic Artin-Schreier theory is \cite[Section 3.7]{stichtenoth2009algebraic}.  If $F$ is a field of characteristic $p$, and $u\in F$ such that $u\neq w^p-w$ for any $w\in F$, and $y$ satisfies $y^p-y=u$, then $F(y)$ is an Artin-Schreier extension of $F$.  If $F$ is the function field of a variety $\mathcal{X}$, and $F(y)$ is the function field of a variety $\mathcal{Y}$, we refer to the induced map $\mathcal{Y} \rightarrow \mathcal{X}$ as an Artin-Schreier cover.  If $\mathcal{X}$ is the projective line in $x$, the curve $\mathcal{Y}$ is defined by an equation of the form $y^p-y=f(x)$ for $f(x)$ a rational function and is called an Artin-Schreier curve.  Note that Artin-Schreier curves have an order-$p$ automorphism defined by $\alpha(x,y)=(x+1,y)$.  We then have that $\mathcal{X}=\mathcal{Y}/\langle\alpha\rangle$.

Artin-Schreier curves have been well-studied for theoretical interest as degree-$p$ Galois covers of the projective line in characteristic $p$.  For example, in characteristic $2$, hyperelliptic curves are Artin-Schreier curves.  These curves have also attracted notice in coding theory.  One reason for this is that the form of Artin-Schreier curves makes it possible to count the points for curves in some families using number theory and geometry.  Also, certain families have been found to contain curves with many points relative to the Weil bound.  Curves with many points are desirable because these can lead to codes with parameters that can be bounded using geometry, but which are relatively long relative to the size of their field of definition. Other connections also exist between Artin-Schreier curves and codes, particularly in determining the weight distribution of codes through counting points on Artin-Schreier curves, see  \cite{VANDERGEER1991256}.

An Artin-Schreier surface is defined over a field $F$ of characteristic $p$ by a equation of the form $y^p-y=f(x,z)$, where $f(x,z)$  is a rational function.  Note that in this paper, we focus on the affine points of such surfaces.  Artin-Schreier surfaces or hypersurfaces are less well-understood. Interesting recent work counts points on some families of Artin-Schreier surfaces \cite{ASpointcounts}, and point counts of Artin-Schreier hypersurfaces were used to determine weights and minimum distance bounds in \cite{AScycliccodes}.  One nice feature of our work here is counting the points on curves in two such families in Sections \ref{sec:Example} and \ref{sec:RelatedExample}.

One important result that we will need to count these points is Hilbert's Theorem 90.  

\begin{theorem}[Hilbert's Theorem 90] Let $L/K$ is an extension of fields with cyclic Galois group generated by $\sigma$, and let $\N_{L/K}$ be the field norm from $L$ to $K$. If $a\in L$, then $\N_{L/K}(a)=1$ exactly when there exists $b\in L$ with $a=\frac{b}{\sigma(b)}$.  
\end{theorem}
For the extension of finite fields $\F_{p^2}/\F_p$, this means that for $a\in\F_{p^2}$, $\N(a)=1$ if and only if there exists $b\in\F_{p^2}$ with $a=\frac{b}{b^{p}}$.  There is also an additive form of Hilbert's Theorem 90.  In the same finite field extension, this can be restated as: for $a\in\F_{p^2}$, $\Tr(a)=0$ if and only if there exists $b\in\F_{p^2}$ with $a=b-b^p$.

\section{Main Theorem}
\label{sec:main theorem}

Let $p$ be an odd prime and $h$ a positive integer. Let $f(x,z)\in\Fph[x,z]$ and let $\calS$ be an affine Artin-Schreier surface defined by $y^p-y=f(x,z)$. For each $\gamma \in \Fph$ such that $f(x,\gamma)$ is non-constant, define the curve
    \begin{equation}\label{eqn:calZ}
        \Zgamma\coloneqq V(y^p-y-f(x,z),z-\gamma).
    \end{equation}

In what follows, we consider evaluation codes defined using affine points of $\calS$. In subsequent sections, we shall restrict our attention to particular families of Artin-Schreier surfaces, but the results here are stated in greater generality. In particular, we place no restriction on the finite field, and \emph{a priori}, no restriction on the polynomial $f$ in the defining equation of $\calS$. However, in order to determine the parameters of the codes we consider, we must impose some geometric conditions on the surfaces $\calS$ as well as on the subset of points $T$ we take as our evaluation set.

As mentioned in the introduction, surfaces can be a natural source of evaluation codes with hierarchical recovery, in this case, by leveraging the fact that $\calS$ comes equipped with fibrations into affine Artin-Schreier curves. In our setup, the points on the curves $\calZ_\gamma$ arising as fibers of the projection map $\pi_z \colon \calS \to \A^1_z$ form the support for the middle codes in our hierarchy. We do not use all points and all fibers to form our evaluation sets, though, since this leads to codes with small rates and (relative) minimum distances. Instead, we keep only the curves $\calZ_\gamma$ that have sufficiently many rational points to guarantee we can recover many erasures. We use a set $\Gamma \subset \Fph$ to give a crude measure of how spread out rational points are on the curves $\calZ_\gamma$, namely by keeping track of those $\gamma \in \Fph$ for which the curve $\calZ_\gamma$ has rational points with many distinct $x$-coordinates. Furthermore, 
we require that our surfaces have rational points that are sufficiently spread out across all of $\calS$, in the sense that we require the size of the set $\Gamma$ to be large.

\begin{theorem} \label{thm: main}
    Let $\calS$ and $\Zgamma$ be defined as above. Let $\eta$ be a positive integer and let $\Gamma\subseteq\Fph$ denote the subset 
    \begin{equation}\label{eqn:Gamma}
        \Gamma \coloneqq \{ \gamma \in \Fph : 
        \#\pi_x(\Zgamma(\Fph))\geq \eta
        \}.
    \end{equation}

    Assume there exist non-negative integers $\nu, \rho_1, \rho_2$, and $\rho_3$ such that
    \begin{enumerate}
    \item $\deg(\mathcal{Z}_{\gamma})\leq \nu$ for all $\gamma\in\Gamma$;
    \item $2 \le \rho_1 \le \eta$, $2\leq\rho_2\leq p$, and $0\leq \rho_3\leq p^2-1$;
    \item there are at least $\rho_3 + 1$ values of $\gamma\in\Gamma$ so that $\#\Zgamma(\F_{p^h})\geq \nu(\eta - \rho_1 + p-\rho_2)+1$.
    \end{enumerate}

    Define an evaluation code $C \coloneqq C(T,V, \eta, \nu)$ with evaluation set
    \begin{equation}\label{eqn:T}
        T=\{(x,y,z)\in \calS(\Fph): z\in\Gamma\},
    \end{equation}
and vector space of polynomials
    \begin{equation}\label{eqn:V}
        V=\langle x^iy^jz^k:
        0\leq i\leq \eta-\rho_1, 
        0\leq j\leq p-\rho_2, 
        0\leq k \leq \rho_3 
        \rangle.
    \end{equation} 
    Then the code $C$ has hierarchical recovery with lower codes supported
    on the fibers of $\pi_{x,z} \colon \calS \to \A^2_{x,z}$ and middle codes supported on the curves $\calZ_\gamma$, with $d_2=\rho_2$, $d_1=\rho_1\rho_2$, $s_2=p-\rho_2+1$, $s_1=(\eta-\rho_1+1)(p-\rho_2+1)$.
    Furthermore, the dimension of the full code is $(\eta-\rho_1+1)(p-\rho_2+1)(\rho_3+1)$.
\end{theorem}

\begin{proof} 
Consider the position $i$ of a vector in the code $C$. This position corresponds to the evaluation of the polynomials in $V$ at some fixed point $P_i \in T$. 

Let $C_{2,i}$ denote the punctured code given by retaining the positions in the code vectors corresponding to the points in the evaluation set with the same $x$- and $z$-coordinates as $P_i$.  Specifically, these are the points in the fiber 
        \begin{equation}
            F_{2,i} \coloneqq (\pi_{x,z}\mid_\calS)^{-1}(\pi_{x,z}(P_i))=\{(\pi_x(P_i),\pi_y(P_i)+\delta,\pi_z(P_i)):\delta\in\F_p\}.
        \end{equation}
    Note that here we are using the finite field property that for any $\delta \in \F_p$, then $y^p-y=(y+\delta)^p-(y+\delta)$, so we can describe the $p$ points in the fiber explicitly.    
    If $g\in V$, then $g|_{F_{2,i}}$ is a univariate polynomial in $y$ of degree at most $p-\rho_2$. Therefore, knowing the evaluation of $g$ at $p-\rho_2+1$ points in $F_{2,i}$ will allow us to interpolate $g$ across all of $F_{2,i}$. Thus we have local recovery with locality $p-(\rho_2-1)$, and any $\rho_2-1$ erasures of coordinates corresponding to points in $F_{2,i}$ can be recovered at the lower level. Therefore the minimum distance of the lower code is at least $d_2=\rho_2$ and $C_{2,i}$ is a $(s_2, d_2)$-locally recoverable code.

     Now consider the punctured code $C_{1,i}$ given by retaining the positions in the code vectors corresponding to the points in the evaluation set with the same $z$-coordinate as $P_i$. These points can be described as the points in the fiber
        \begin{equation}
            F_{1,i} \coloneqq (\pi_z\mid_\calS)^{-1}(\pi_z(P_i)),
        \end{equation} 
        or, equivalently, as the $\Fph$-points on $\mathcal{Z}_{\pi_z(P_i)}\colon y^p-y=f(x,\pi_z(P_i))$, a curve in the plane $z=\pi_z(P_i)$ in $\mathbb{A}^3$.

    We want to show that the punctured code $C_{1,i}$ has minimum distance at least $\rho_1\rho_2$.  To do so, we will determine a lower bound on the number of coordinates that uniquely determine a vector in the code.

    Let $\tilde{g}\in V$ and consider its restriction to the fiber $F_{1,i}$. Write 
        \[g(x,y):= \tilde{g}(x,y, \pi_z(P_i)) = g_0(x)+g_1(x)y+\dots g_{p-\rho_2}(x)y^{p-\rho_2},\]  
    where $g_j\in\Fph[x]$ with $\deg(g_j)\leq \eta-\rho_1$ for each $0\leq j\leq p-\rho_2$. 
    We consider the process of polynomial interpolation in the setting of this multivariate polynomial $g$.

   Let $U=\pi_x(F_{1,i})$, so $U$ is the set of $x$ values represented in $F_{1,i}$. For each $\alpha\in U$, let $W_\alpha=\pi_{x,y}(\{P \in F_{1,i} : \pi_x(P)=\alpha\})$, consisting of pairs $(\alpha,y)$ satisfying the equation $y^p - y = f(\alpha, \pi_z(P_i))$, so that $\#W_\alpha=p$.
    The main idea here is to find each $g_j(x)$ by standard univariate polynomial interpolation, which requires finding $g_j(\alpha)$ for at least $\eta-\rho_1+1$ values of $\alpha\in U$.  To do this, note that
   \[g(\alpha,y)=g_0(\alpha)+g_1(\alpha)y+\dots g_{p-\rho_2}(\alpha)y^{p-\rho_2}\] is a univariate polynomial, and its coefficients can be interpolated given the values of $g(\alpha,y)$ on at least $p-\rho_2+1$ distinct $y\in W_{\alpha}$.  
 
   To enable this layered process of interpolation on $g(x,y)$ we need $\eta-(\rho_1-1)$ values of $\alpha \in U$ for which we know the evaluations of $g(\alpha,y)$ on at least $ p - (\rho_2-1)$ points in $W_\alpha$.

    For a fixed $\alpha \in U$, determining $g_j(\alpha)$ for all $j$ will fail only if there are at least $\rho_2$ erasures at indices corresponding to points of the form $(\alpha,y,\pi_z(P_i))\in T$. 
    Since $P_i \in T$, we have $\pi_z(P_i) \in \Gamma$ and thus $\#U \ge \eta$ by definition of $\Gamma$. The next step, determining the coefficients of $g_j(x)$ for each $j$, will fail only if there are at least $\rho_1$ values of $\alpha\in U$ so that we failed to find $g_j(\alpha)$. 
    Thus there must be at least $\rho_1\rho_2$ missing evaluations of $g(x,y)$ in $F_{1,i}$ for the recovery to fail. This guarantees that minimum distance of the middle code is at least $d_1 = \rho_1\rho_2$.

    Finally, we consider the full code.  We prove that the evaluation map is injective for $V$ and $T$, thus determining the dimension of the code.

   Let $\tilde{g}\in V$, let $\gamma_0 \in \Gamma$, and consider $g(x,y) \coloneqq \tilde{g}(x,y,\gamma_0)$.  The total degree of $g$ is at most $\eta-\rho_1 + p - \rho_2$.  The degree of the irreducible plane curve $\calZ_{\gamma_0}$ defined by $y^p-y=f(x,\gamma_0)$ is at most $\nu$, by assumption. By Bezout's theorem, if $g\not \equiv 0$, then $g$ vanishes for at most $s \coloneqq \nu(\eta- \rho_1 + p - \rho_2)$ points in $\calZ_{\gamma_0}$. Since $\deg_z(\tilde{g})\leq \rho_3$, there can be at most $\rho_3$ values of $\gamma\in \Gamma$ for which $\tilde{g}|_{\Zgamma} \equiv 0$.
    By assumption, there are at least $\rho_3+1$ values of $\gamma\in \Gamma$ where $\#\Zgamma(\Fph)\geq \nu( \eta-\rho_1 + p-\rho_2)+1 > s$. Thus there must be at least one value $\gamma_0 \in \Gamma$ for which $g\not \equiv0$, and hence $\tilde{g}$ must be nonzero.    
    This shows the evaluation map is injective, and therefore the dimension of $C(T,V)$ is $k=\dim(V) = (\eta-\rho_1+1)(p-\rho_2+1)(\rho_3+1)$. 
\end{proof}

\begin{remark} \label{rem: mindist} In Theorem \ref{thm: main}, we proved that the evaluation map is injective. This proof can also be used to give a lower bound on the minimum distance of the full code $C$. In particular, we showed that if $\tilde{g} \in V$ is not identically $0$ when restricted to $\mathcal{Z}_{\gamma}$ for $\gamma \in \Gamma$, then $\tilde{g}\mid_{\mathcal{Z}_\gamma}$ vanishes for at most $s = \nu(\eta- \rho_1 + p - \rho_2)$ points of $\mathcal{Z}_\gamma$. Hence it does not vanish on at least $\#\mathcal{Z}_{\gamma}(\F_{p^h})-s$ points. At this level of generality, we are therefore only able to say that the minimum distance of $C$ is $d \ge \#\mathcal{Z}_{\gamma}(\F_{p^h}) - s \ge 1$, which is not a particularly meaningful lower bound. However, for a fixed surface $\calS$, one can give a more precise bound on the number of points on Artin-Schreier curve fibers $\calZ_\gamma$, and hence a non-trivial lower bound on the minimum distance. See Theorem \ref{thm: surface} for an explicit example. 
\end{remark}

\section{Explicit Example}\label{sec:Example}
We now consider the surface $\calS$ defined by 
    \begin{equation} \label{eqn: extremeS} y^p-y=x^{p+1}z^2+x^2z^{p+1},\end{equation} 
where $p$ is an odd prime. 
This surface has Artin-Schreier curves as fibers above $z\neq 0$, and so we will define an evaluation code $C$ with hierarchical locality on a subset of the affine points of $\calS(\F_{p^2})$ using the construction in Section \ref{sec:main theorem}.  Let
\begin{equation} \label{eqn: exT}
    T = \{(x,y,z) \in \calS(\F_{p^2}):  z\neq 0\}.
\end{equation} 
For any integers $2\leq \rho_1,\rho_2\leq p$, let
    \begin{equation} \label{eqn: exV}
        V = V_{\rho_1,\rho_2}=\langle x^iy^jz^k: 0\leq i\leq p-\rho_1, 0\leq j\leq p-\rho_2, 0\leq k \leq p^2-2p\rangle.
    \end{equation}

\subsection{Point counting} In order to prove that $C \coloneqq C(V_{\rho_1, \rho_2},T)$ has hierarchical locality (see Theorem \ref{thm: surface}) we first need to count the number of points on $\calS$ over $\F_{p^2}$. Counting these points and specifying the evaluation set will allow for the computation of the parameters of the code.

\begin{lemma}\label{lem:sum in Fp}
    For all $a\in\F_{p^2}^\times$, we have $\frac{1}{a^{p-1}}+a^{p-1}\in\F_p$.
\end{lemma}
\begin{proof} We claim that $\frac{1}{a^{p-1}}+a^{p-1}\in\F_p$ since it is fixed by the Galois action on $\F_{p^2}$ over $\F_p$. Indeed, we have
\begin{equation*}
\left(\frac{1}{a^{p-1}}+a^{p-1}\right)^p = \frac{1}{a^{p^2-p}}+a^{p^2-p}=\frac{1}{a^{1-p}}+a^{1-p}=a^{p-1}+\frac{1}{a^{p-1}},
\end{equation*}
as required.
\end{proof}
\begin{lemma}\label{lem:num elts of form}
   There are %$2+\frac{p-1}{2}$ 
   $\frac{p+3}{2}$ elements $u\in\F_p$ that can be written as $u = \frac{1}{a^{p-1}}+a^{p-1}$ for some $a\in\F_{p^2}^\times$.
\end{lemma}

\begin{proof}
    By Hilbert's theorem 90, the elements of the form $a^{p-1}$ are precisely the elements of $\F_{p^2}$ of norm $1$ over $\F_p$.
   Since the norm map $N_{\F_{p^2}/\F_p}$ is a surjective homomorphism on the multiplicative groups, its kernel has size $(p^2-1)/(p-1) = p+1$, hence there are $p+1$ elements of norm $1$ in $\F_{p^2}$. 

    Suppose that $u \in \F_p$ and for some $a^{p-1}$, $b^{p-1}$ in $\F_{p^2}$ we have 
    $$u=a^{p-1}+\frac{1}{a^{p-1}}=b^{p-1}+\frac{1}{b^{p-1}}.$$ 
    Then $a^{p-1}$ and $b^{p-1}$ are both roots of the polynomial $T^2-uT+1$, hence we must have $a^{p-1}=\frac{1}{b^{p-1}}$. Note that $a^{p-1}= \frac{1}{a^{p-1}}$ if and only if  $a^{p-1}=\pm 1$. The remaining $p-1$ elements of $\F_{p^2}$ with norm 1 can be paired as $(a^{p-1},\frac{1}{a^{p-1}})$, giving rise to $\frac{p-1}{2}$ distinct 
    $$u=a^{p-1}+\frac{1}{a^{p-1}}.$$ 

Therefore, in total, we have $2 + \frac{p-1}{2} = \frac{p+3}{2}$ distinct $u \in \F_p$ which can be written in the desired form.
\end{proof}

\begin{prop} \label{lemma: pointcount}
    For the surface $S$ defined in \eqref{eqn: extremeS},
    \begin{equation}\label{eqn:point_count}
        \#S(\F_{p^2}) = p(2(p-1)^3+2(p-1)^2+2p^2-1) =2p^4 - 2p^3 + 2p^2 - p.
    \end{equation}
\end{prop}
\begin{proof}
    We again apply additive Hilbert's Theorem 90 to conclude that there exists $y\in\F_{p^2}$ with $y^p-y=x^{p+1}z^2+x^2z^{p+1}$ if and only if $\Tr_{\F_{p^2}/\F_p}(x^{p+1}z^2+x^2z^{p+1})=0$.  This yields
    \begin{align*}
        \Tr_{\F_{p^2}/\F_p}(x^{p+1}z^2+x^2z^{p+1}) 
        &=(x^{p+1}z^2+x^2z^{p+1})^p+(x^{p+1}z^2+x^2z^{p+1}) \stepcounter{equation}\tag{\theequation}\label{myeq1} \\
        &=x^{p^2+p}z^{2p}+z^{p^2+p}x^{2p}+x^{p+1}z^2+x^2z^{p+1}\\
        &=x^{p+1}(z^{2p}+z^2)+z^{p+1}(x^{2p}+x^2).
    \end{align*}
    This quantity certainly equals 0 when either $x=0$ or $z=0$. There are $2p^2-1$ such pairs of values $(x,z) \in \F_{p^2} \times \F_{p^2}$.

    If both $x\neq 0$ and $z\neq 0$, $\Tr_{\F_{p^2}/\F_p}(x^{p+1}z^2+x^2z^{p+1})=0$ if and only if 
       % \begin{equation}
        %    \frac{x^2+x^{2p}}{x^{p+1}}=-\frac{z^2+z^{2p}}{z^{p+1}},
       % \end{equation} 
    %which simplifies to 
        \begin{equation}\label{eqn:simplified}
            \frac{1}{x^{p-1}}+x^{p-1}=-\frac{1}{z^{p-1}}-z^{p-1}.
        \end{equation} 

Now, for each $u \in \F_p$, consider those pairs $(x,z)$ that satisfy 
\begin{equation} \label{eqn: tracezero}
   \frac{1}{x^{p-1}}+x^{p-1}=-\frac{1}{z^{p-1}}-z^{p-1} = u.  
\end{equation}

%Thus, we must simultaneously solve $c + c^{-1} = \alpha$ and $d + d^{-1} = -\alpha$ with $c = x^{p-1}$ and $d = z^{p-1}$. 

When $u = 2$, as noted in the proof of Lemma \ref{lem:num elts of form}, this implies $x^{p-1} = -z^{p-1} = 1$. There are $p-1$ distinct values of $x$ such that $x^{p-1} = 1$ and $p-1$ distinct values of $z$ with $z^{p-1} = -1$, giving rise to $(p-1)^2$ pairs $(x,z)$. The same argument gives $(p-1)^2$ pairs $(x,z)$ when $u = -2$.

By Lemma \ref{lem:num elts of form}, there are $\frac{p-1}{2}$ additional values of $u \in \F_p$ for which equation~\eqref{eqn: tracezero} has solutions. %For each such $u$, there are 2 values of $x^{p-1}$ and 2 values of $z^{p-1}$. giving $p-1$ values of $x$ and $p-1$ values of $z$ respectively. 
Moreover, for a fixed $u$, the solutions to equation~\eqref{eqn: tracezero} come in sets of four
\[\left\{(\alpha,\beta),(\alpha,\beta^{-1}),(\alpha^{-1},\beta),(\alpha^{-1},\beta^{-1})\right\}.\]
For any given $\alpha,\beta$, the equations $x^{p-1}=\alpha$ and $z^{p-1}=\beta$ each have $p-1$ solutions in $\F_{p^2}$. Thus there are $4(p-1)^2$ pairs of values $(x,z)$ for each of these $\frac{p-1}{2}$ elements $u$. Summing over all $u$, we have $2(p-1)^3+2(p-1)^2$ total pairs $(x,z)$ with both $x\neq 0$ and $z\neq 0$.  
  
Finally, for each of these pairs $(x,z)$ with $\Tr(x^{p+1}z^2+x^2z^{p+1})=0$ there are $p$ different $y$-values, yielding  \[ \#S(\F_{p^2}) = p(2(p-1)^3+2(p-1)^2+2p^2-1) =2p^4 - 2p^3 + 2p^2 - p,\]
as required.
\end{proof}

\subsection{The evaluation code} We can now prove that the evaluation code $C(V_{\rho_1, \rho_2}, T)$ has hierarchical locality. Recall that the surface $\calS$ has defining equation $y^p-y=x^{p+1}z^2+x^2z^{p+1}$, the subset of points  is $T = \{(x,y,z) \in \calS(\F_{p^2}):  z\neq 0\}$, and for integers $2\leq \rho_1,\rho_2\leq p$, we consider the functions $V = V_{\rho_1,\rho_2}=\langle x^iy^jz^k: 0\leq i\leq p-\rho_1, 0\leq j\leq p-\rho_2, 0\leq k \leq p^2-2p\rangle.$ 
\begin{theorem} \label{thm: surface}
    The code $C(V_{\rho_1,\rho_2},T)$ has hierarchical locality. 
    In particular,
    \begin{enumerate}
        \item the lower code has length $n_2 = p$, dimension $k_2 \leq s_2=p-\rho_2+1$ and minimum distance at least $ d_2=\rho_2$
        \item the middle code has length $n_1\leq 2p^2-p$, dimension $k_1 \leq s_1=(p-\rho_2+1)(p-\rho_1+1)$, and minimum distance at least $d_1=\rho_2\rho_1$, and 
        \item the full code has length $n =2p^4-3p^3+2p^2-p$, dimension $k =(p-\rho_2+1)(p-\rho_1+1)(p^2-2p+1)$ and minimum distance $d\geq (\rho_1+\rho_2-3)p+(\rho_1+\rho_2)$.
    \end{enumerate}
\end{theorem}

\begin{proof}
We will apply Theorem \ref{thm: main}. To do so, we will show that it suffices to take $\eta = p$, $\nu = p+1$, and $\rho_3 = p^2 - 2p$. First,  as noted in the proof of Lemma \ref{lemma: pointcount}, if $\gamma^{p-1}=\gamma^{1-p}$, then there are $p-1$ nonzero values of $x$  with $\Tr_{\F_{p^2}/\F_p}(\gamma^2x^{p+1}+\gamma^{p+1}x^2)=0$ otherwise, there are $2(p-1)$ such $x$-values.
    In addition, if we include the case when $x=0$, we find
    \begin{equation} \label{eqn: numxvalues}
        \#\pi_x(\mathcal{Z}_\gamma(\F_{p^2}))\in\{p,2p-1\}.
    \end{equation}
So we may take $\Gamma = \{\gamma \in \F_{p^2} \setminus \{ 0 \}: \#\pi_x(\mathcal{Z}_\gamma(\F_{p^2})) \ge p\}$ and $\eta = p$.  Since for all $\gamma \in \Gamma$, we have $\deg(\mathcal{Z}_\gamma)  \le p+1$ we may take $\nu = p +1$. 

Now we verify that it suffices to take $\rho_3 = p^2 - 2p$. We will show that there are at least $\rho_3 + 1 = p^2 - 2p +1$ values of $\gamma \in \Gamma$ such that \begin{equation}
     \# \mathcal{Z}_\gamma(\F_{p^2}) \ge 2p^2 - p > (p+1)(p-2 + p-2) +1  \ge \nu(\eta - \rho_1 + p - \rho_2) + 1.
\end{equation} By the discussion above and the proof of Lemma~\ref{lemma: pointcount} we note that \begin{equation} \#\{\gamma\in\Gamma: \#\mathcal{Z}_{\gamma}(\F_{p^2})\geq 2p^2 - p\} \end{equation} represents the number of nonzero $\gamma \in \F_{p^2}$ such that $\gamma^{p-1} \ne \gamma^{1-p}$. Indeed, there are $p^2 -1$ nonzero $\gamma \in \F_{p^2}$ and the number of solutions to $\gamma^{p-1} = \gamma^{1-p}$ is $2(p-1)$, coming from $p-1$ solutions when $\gamma^{p-1} = 1$ and $p-1$ when $\gamma^{p-1} = -1$. Hence the total number of desirable $\gamma$ is $p^2 - 1 - 2(p-1) = p^2 - 2p + 1$. 

Finally, we verify the parameters of the lower, middle, and full codes. Let $i$ denote the position in the code corresponding to the point $P_i \in S(\F_{p^2})$ and let $\gamma_i \coloneqq \pi_z(P_i)$.

(1) Recall that the lower code $C_{2,i}$ represents the punctured code given by retaining the positions in the code vectors corresponding to the points in the fiber 
$F_{2,i} = (\pi_{x,z}\mid_\calS)^{-1}(\pi_{x,z}(P_i))$ which is a set of size $p$. Hence $C_{2,i}$ has length $n_2 = p$, and the dimension and minimum distance are immediate consequences of the proof of Theorem \ref{thm: main}

(2) The middle code $C_{1,i}$ is given by retaining the positions in the code vectors corresponding to the points in the evaluation set with the same $z$-coordinate as $P_i$. These points can be described as the $\F_{p^2}$-points on the curve $\mathcal{Z}_{\gamma_i}$. By equation \eqref{eqn: numxvalues}, such a curve can have at most $p(2p-1) = 2p^2 - p$ points, hence the length of $C_{1,i}$ at most $2p^2 - p$. The upper bound on the dimension and the lower bound on the minimum distance are proved in Theorem \ref{thm: main}. 

(3) The length of the full code $C$ is $n =  \#S(\F_{p^2}) - p^3 = 2p^4 - 3p^3 - 2p - 2$ by Proposition \ref{lemma: pointcount} and the fact that there are $p^3$ points on $S$ with $z =0$. The dimension is immediate from the definition of $V_{\rho_1, \rho_2}$. Remark \ref{rem: mindist} following the proof of Theorem \ref{thm: main} shows that we can obtain the following lower bound on the minimum distance,
\begin{align*}d &\ge \#\mathcal{Z}_\gamma(\F_{p^2}) - \nu(\eta- \rho_1 + p - \rho_2) \\
& \ge  (2p^2-p)-(p+1)(2p-(\rho_1+\rho_2)) \\
 &=(\rho_1+\rho_2-3)p+(\rho_1+\rho_2).
 \end{align*} 
Since $\rho_1, \rho_2 \ge 2$ by assumption, this shows that $C$ has a positive, nontrivial minimum distance.
\end{proof}

\begin{remark}
Note that this construction gives a maximum asymptotic information rate of about $\frac{1}{2}$.  The relative minimum distance does not approach 1 for any choice of $\rho_1$, and $\rho_2$. If a code with high relative minimum distance were desired, we could alter the original construction to trade off rate and minimum distance.  Here, we would let $\rho_1,\rho_2, \rho_3$ be integers with $2\leq \rho_1,\rho_2\leq p$, $2\leq \rho_3\leq p^2-2p$, and let \[V_{\rho_1,\rho_2, \rho_3}=\langle x^iy^jz^k: 0\leq i\leq p-\rho_1, 0\leq j\leq p-\rho_2,  0\leq k \leq p^2-2p-\rho_3\rangle.\] Defining the corresponding code $C(T,V_{\rho_1,\rho_2,\rho_3})$ would let us choose parameters to take smaller rate in exchange for large relative minimum distance. We do not explicitly determine the parameters for this construction here but it is possible using the techniques of this paper.
\end{remark}

\begin{remark}
    
One can also consider the surfaces \begin{equation} \label{eqn: altsurface} \calS_\lambda : y^p -y = x^{p+1}z^2 + x^2z^{p+1} - \lambda
\end{equation} for $\lambda \in \F_{p^2}$. If $\Tr(\lambda) = 0$, then the number of $\F_{p^2}$ rational points of $\calS_\lambda$ and the results of Theorem~\ref{thm: surface} are unchanged; indeed, the curves are isomorphic by a translation of $\lambda$ on the $y$-axis. However, a computation in \texttt{Magma} \cite{magma} shows that for all $p \le 50$ if $\Tr(\lambda) \ne 0$, then for any $\gamma \in \F_{p^2}^\times$ the Artin-Schreier curve \begin{equation} \calZ_{\gamma} : y^p - y = \gamma^2x^{p+1} + \gamma^{p+1} x^2 -\lambda \end{equation} 
has $\#\pi_x(\mathcal{Z}_\gamma(\F_{p^2}))\in\{0, p-1, 2p\}$ and thus $\# \calZ_\gamma(\F_{p^2}) \in \{ 0, p^2-p, 2p^2 \}$. Moreover, \begin{equation} \#\calS_\lambda(\F_{p^2})  \in \{(p^2-p)(p-1)^2, (p^2-p)(p-1)^2 + 2p^2(p-1) \}.
\end{equation}
We may take $\Gamma = \{ \gamma \in \F_{p^2}^\times :\#\pi_x(\mathcal{Z}_\gamma(\F_{p^2})) \ge p-1 \}$. This forces $\eta = p-1$, and as before, $\nu = p+1$. Since $2p^2 > (p+1)(p-1-2+p-2) \ge \nu(\eta - \rho_1 + p - \rho_2)$ and again a computation for $p \le 50$ gives $\#\{\gamma \in \Gamma: \calZ_\gamma(\F_p^2) \ge \nu(\eta - \rho_1 + p - \rho_2) \} \ge p-1$ in this case we get $\rho_3 \ge p-2$. That is, the evaluation code $C(T, V_{\rho_1, \rho_2})$ with evaluation set  and space of functions \begin{eqnarray} T&=&\{(x,y,z)\in \calS(\F_{p^2}): z\in\Gamma\}  \\
 V_{\rho_1, \rho_2}& =& \langle x^i y^j z^k : 0 \le i \le p - \rho_1 -1, 0 \le j \le p-\rho_2, 0\le k \le p-2 \rangle, \end{eqnarray} respectively, satisfies the hypotheses of Theorem \ref{thm: main} with $\Gamma, \eta, \nu$, and $\rho_3$ as specified above. Thus $C(T,V)$ has hierarchical recovery with lower codes supported on the fibers of the Artin-Schreier cover, middle codes supported on the curves $\calZ_\gamma$, and parameters as in Theorem \ref{thm: main}.  This is a promising direction for further exploration.
\end{remark} 

\section{A related example} \label{sec:RelatedExample}  We now describe a code with hierarchical locality for an Artin-Schreier surface which does not meet the conditions of Theorem \ref{thm: main}.  This example is valuable in illustrating that the difficult point counting in Section \ref{sec:Example} depends deeply on the equation of the surface, and that applying the same basic ideas to different Artin-Schreier surfaces can generate codes with very different parameters and properties.
Let $p$ be an odd prime. Let $\calS$ be the smooth affine surface defined by $y^p - y = x^{p+1} z^{p+1}$ over $\F_{p^2}$, where smoothness follows from the fact that $\partial_y$ is always nonzero. We consider an evaluation code $C$ on this surface and determine its parameters. In contrast to the examples from Section \ref{sec:Example}, however, the rational points on $\calS$ are not well-distributed, so we will not apply Theorem \ref{thm: main} directly, but will still construct a code with hierarchical recovery.

\begin{lemma} \label{lem: ptcount1stex}
    Let $\calS$ be the smooth affine surface defined over $\F_{p^2}$ by $y^p-y = x^{p+1}z^{p+1}$. Then $\#\calS(\F_{p^2}) = p^3 + p^3 - p = 2p^3 - p$. 
\end{lemma}

\begin{proof} The additive version of Hilbert's Theorem 90 tells us that $c \in \F_{p^2}$ has trace zero if and only if $c = \gamma^p - \gamma$ for some $\gamma \in \F_{p^2}$. Thus $y^p - y= x^{p+1}z^{p+1}$ if and only if $\Tr(x^{p+1}z^{p+1}) = 0$. Since $N_{\F_{p^2}/\F_p}(\alpha) = \alpha \alpha^p = \alpha^{p+1} \in \F_p$,  we have $\Tr(x^{p+1}z^{p+1}) = 2x^{p+1}z^{p+1}$, which vanishes precisely when $x = 0$ or $z = 0$. When $x = 0$, there are $p^2$ possible choices for $z$, and there are $p$ solutions to $y^p - y = 0$, i.e., the elements of $\F_p$. Thus from $x = 0$ we find $p^3$ points in $\calS(\F_{p^2})$. The same argument holds swapping the roles of $x$ and $z$. This includes the $p$ points on $\calS$ with $x = z = 0$ twice, thus in total, we find $\#\calS(\F_{p^2}) = p^3 + p^3 - p = 2p^3 - p$. 
\end{proof}

\subsection{An evaluation code} We now consider an evaluation code on the (affine) points of $\calS$.

\begin{theorem} \label{thm: secondsurface} Let $\rho_1$ and $\rho_2$ be positive integers such that $2\leq\rho_2\leq p$ and $p\leq\rho_1\leq p^2$. Let $V := \langle y^i x^j, y^k z^\ell \mid 0 \le i,k \le p-\rho_2, 0 \le j, \ell \le p^2 -\rho_1 \rangle$.  The code $C \coloneqq C(\calS(\F_{p^2}),V)$ is a locally recoverable code with hierarchical locality with lower codes supported on the fibers of the projection $\pi_{x,z} \colon \calS \to \A^2_{x,z}$ and middle codes supported on the curves arising as the intersection of $\calS$ with the plane $x=0$ or $z=0$. In particular, 
\begin{enumerate}
\item The lower codes have length $n_2 = p$, dimension at most $s_2 = p-\rho_2$ and minimum distance at least $d_2 = \rho_2$.
\item The middle codes have length $n_1=p^3$, dimension at most $s_1=(p^2-\rho_1+1)(p-\rho_2+1)$, and minimum distance at least $d_1 = \min(\rho_1 \rho_2, \, p(\rho_1+\rho_2)-p^2)$.
\item The full code has length $n=2p^3-p$ and dimension $k=2p^3-2p^2(\rho_2-1)-p(2\rho_1-1)+(2\rho_1-1)(\rho_2-1)$, and minimum distance $d\geq \min(\rho_1 \rho_2, \, p(\rho_1+\rho_2)-p^2)$.
\end{enumerate}
\end{theorem}

\begin{proof} Much of the proof will follow a similar argument to that of Theorem \ref{thm: main}. However, since the space of polynomials $V$ defined above is a proper subset of the vector space used in Theorem \ref{thm: main}, some modifications are needed. As in Theorem \ref{thm: main}, a position $i$ of a vector in the code $C$ corresponds to the evaluation of polynomials in $V$ at a fixed point $P_i \in S(\F_{p^2})$.

(1)  The lower code $C_{2,i}$ is the punctured code given by retaining the position in the code vectors corresponding to points in $S(\F_{p^2})$ with the same $x$- and $z$-coordinates as $P_i$, i.e., points with either $x$-coordinate or $z$-coordinate equal to $0$. Thus if $g \in V$, then the restriction of $g$ to this set of points yields a univariate polynomial in $y$ of degree at most $p - \rho_2$. As in the proof of Theorem \ref{thm: main}, this means we have local recovery with locality $p-(\rho_2 - 1)$ and minimum distance at least $d_2 = \rho_2$ so that $C_{2,i}$ is a $(s_2, d_2)$-locally recoverable code.

(2) There are two middle codes, one supported on $V(y^p-y - x^{p+1}z^{p+1}, x)$ and the other on $V(y^p-y - x^{p+1}z^{p+1}, z)$. Each has length $p^3$, which represents the number of $\F_{p^2}$-rational points on these curves. We will consider the middle code supported on $Z \coloneqq V(y^p-y - x^{p+1}z^{p+1}, z)$, since the argument for the other is identical. Let $\tilde{g} \in V$ and assume that its restriction to $Z$ is not identically zero. We denote this function by $g(x,y)$. To determine a lower bound on the minimum distance, we consider the maximum number of zeros $g$ can have on $Z$. Observe that the degree of $g$ is at most $p^2 - \rho_1 + p - \rho_2$ and the degree of $Z$ is $p$. Let $\calC_g$ be the curve defined by $V(g(x,y), z)$. 

There are two cases, depending on whether $\calC_g$ shares an irreducible component with $Z$ or not. In the case when it does not, then by Bezout's theorem, there are at most $p^3+p^2-p(\rho_1+\rho_2)$ zeros of $g$ on $Z$.  Since $Z(\F_{p^2}) = p^3$, there must be at least $p(\rho_1+\rho_2) - p^2$ points where $g$ does not vanish.  
In the case where the two curves share irreducible components, we may still give a bound on the number of zeros of $g$ on $Z$. Since $\deg_y(g) \le p-\rho_2$, the largest number of irreducible components $\calC_g$ can share with $Z$ is $p - \rho_2$. Each such component has $p^2$ rational points (since $y$ and $z$ are fixed, but $x$ varies over all elements of $\F_{p^2}$). On the remaining $\rho_2$ irreducible components of $Z$, there can be at most $\deg_x(g) \le p^2 - \rho_1$ points. Thus, in total, the maximum number of zeros $g$ can have on $Z$ is $(p-\rho_2)p^2 + \rho_2(p^2 - \rho_1) = p^3 - \rho_1\rho_2$, and thus there must be at least $\rho_1\rho_2$ points on which $g$ does not vanish. The minimum distance is therefore at least $d_1 = \min ( \rho_1 \rho_2, \, p(\rho_1+\rho_2) - p^2)$, which is positive since $\rho_1\rho_2 \ge 2p > 0$ and $\rho_1+\rho_2\geq 2+p > p$. This also gives injectivity of the evaluation map on the points of $Z$, and thus the dimension of the middle code is at most $s_1 = (p^2-\rho_1+1)(p-\rho_2+1)$, which follows from the definition of $V$.
    
(3) The length of the full code $C$ is equal to $\#\calS(\F_{p^2}) = 2p^3 - p$ by Lemma \ref{lem: ptcount1stex}. The minimum distance of $C$ can be bounded below as follows. Suppose that $\tilde{g}\in V$. Note that $\tilde{g}$ cannot be the zero function when restricted to the union of the planes $x=0$ and $z=0$ since then $xz \mid \tilde{g}$, which is impossible since $\tilde{g} \in V$.
Without loss of generality, assume that its restriction to $Z$ is not identically zero. Then the argument in (2) above shows that the minimum distance of the full code is $d\geq \min ( \rho_1 \rho_2, \, p(\rho_1+\rho_2) - p^2)$.  Since $\rho_1\rho_2 \ge 2p > 0$ and $\rho_1+\rho_2\geq 2+p > p$, this also implies that the evaluation map of $V$ on points of $\calS(\F_{p^2})$ is injective, so the dimension of the code is the same as that of $V$, namely $k = \dim(V) = 2(p-\rho_2+1)(p^2-\rho_1+1) - (p-\rho_2+1)$. \end{proof}

\begin{subsection}{Applying Theorem \ref{thm: main} to a related family }
Let $\lambda \in \F_{p^2}$ and consider the surface \begin{equation} \label{eqn: Slambda} \calS_{\lambda} : y^p - y  = x^{p+1}z^{p+1} - \lambda.\end{equation} If $\Tr(\lambda) = 0$, then Theorem \ref{thm: secondsurface} applies unchanged. If $\Tr(\lambda) \ne 0$, then the rational points on $\calS_\lambda$ are better distributed and we can apply Theorem \ref{thm: main} to construct an evaluation code with hierarchical locality. As always, we must first count points. 

\begin{lemma} Let $\gamma \in \F_{p^2}^\times$ such that $\Tr(\lambda) \ne 0$. Then the curve $\calZ_{\gamma}: y^p - y = x^{p+1}\gamma^{p+1} - \lambda$ has $\#\calZ_\gamma(\F_{p^2}) = p(p+1)$. Moreover, the surface $S_\lambda$ defined by equation \eqref{eqn: Slambda} has $\#\calS_\lambda(\F_{p^2}) = p(p+1)(p^2-1) = p^4 + p^3 - p^2 - p$.
\end{lemma} 
\begin{proof} 
To determine the number of points on the curve $\calZ_\gamma$, we must consider the solutions to $\Tr(x^{p+1}\gamma^{p+1} - \lambda) = 0$, which is equivalent to determining the roots  of $2\gamma^{p+1} x^{p+1} = \lambda^p + \lambda$ in $\F_{p^2}$. Since $\gamma \ne 0$ and $\Tr(\lambda) \ne 0$, this gives $x^{p+1} = \frac{\lambda^p + \lambda}{2\gamma^{p+1}}$, where the quantity on the right hand side is an element of $\F_p$. Recall that an equation of the form $x^m - c$ has a root in a finite field $\F_q$ if and only if $\gcd(m, q-1)$ divides $(q-1)/\text{ord}(c)$. In our case, $\gcd(p+1, p^2-1) = p+1$, and $(p^2 - 1)/\text{ord}(c)$ is a multiple of $p+1$ for any $c \in \F_p^\times$ since $\text{ord}(c) \mid p-1$. Moreover, the polynomial $x^{p+1} - c$ splits completely in $\F_{p^2}$ since $\F_{p^2}$ contains the $(p+1)$st roots of unity. Thus the curves $\calZ_\gamma$ have $p(p+1)$ points over $\F_{p^2}$ for all $\gamma \in \F_{p^2}^\times$ and the surface $S_\lambda$ therefore has $p(p+1)(p^2-1)$ points over $\F_{p^2}$.  
\end{proof} 
\begin{theorem}
Let $\lambda \in \F_{p^2}$ with $\Tr(\lambda) \ne 0$ and let $\calS_\lambda$ be defined as in equation \eqref{eqn: Slambda}. Let  \[ T = \{(x,y,z) \in \calS_\lambda(\F_{p^2}):  z\neq 0\},\] and for integers $2\leq \eta\leq p+1$, $2\leq \rho_1 \le \eta$, $2 \le \rho_2\leq p$, with $\eta \le \rho_1 + \rho_2$, we consider the functions \[ V = V_{\rho_1,\rho_2}=\langle x^iy^jz^k: 0\leq i\leq \eta-\rho_1, 0\leq j\leq p-\rho_2, 0\leq k \leq p^2-2\rangle.\] Then the code $C(T, V_{\rho_1, \rho_2})$ has hierarchical locality with hierarchical parameters $n_1=p^2+p$, $s_1=(\eta-\rho_1+1)(p-\rho_2+1)$, and $d_1=\rho_1\rho_2$, $n_2=p$, $s_2=p-\rho_2+1$, and $d_2=\rho_2$.  The full code has length $n=p(p+1)(p^2-1) = p^4 + p^3 - p^2 - p$ and dimension $k=(\eta-\rho_1+1)(p-\rho_2+1)(p^2-1)$.
\end{theorem} 

\begin{proof} We apply Theorem \ref{thm: main} with $\nu = p+1$, $\rho_3 = p^2 - 2$, and $\Gamma = \{ \gamma \in \F_{p^2}^\times : \# \pi_x(\calZ_\gamma(\F_{p^2})) \ge \eta \}$. Since the degree of the curve $\calZ_\gamma$ is at most $ p+1$, it suffices to take $\nu = p+1$. Since $\#\pi_x(\calZ_\gamma(\F_{p^2})) =p+1$ for all $\gamma \ne 0$, we may take $\eta \le p+1$. By assumption, $\eta \le \rho_1 + \rho_2$, thus it suffices to take $\rho_3 = p^2 - 2$ since there are $\rho_3 + 1 = p^2 - 1$ values of $\gamma \in \F_{p^2}^\times$ such that \begin{equation} \# \calZ_\gamma(\F_{p^2}) = p^2 + p \ge (p+1)(\eta - \rho_1 + p - \rho_2) = \nu(\eta - \rho_1 + p - \rho_2). 
\end{equation}  The parameters of the code follow immediately from Theorem \ref{thm: main}.
\end{proof}
\end{subsection}

\section{Conclusions and Next Directions}\label{sec:Conclusion}
In this paper, we use geometric structures in Artin-Schreier surfaces to define HLRCs and look closely at two examples.  Artin-Schreier surfaces were good candidates for this exploration because the natural degree-$p$ cover of $\mathbb{P}^2$ is unramified on the affine points of the curve, and the arithmetic structure of the fibers of this map are well-understood.  These fibers are just the intersections of lines of constant $x$ and $z$ with the surface, and they always have $p$ points of the form $(a, b+y,c)$ where $y$ varies over $\F_p$.  Similarly, the intersections of these surfaces with planes of constant $z$ will often be Artin-Schreier curves, which have a particularly nice arithmetic structure and are amenable to point counting.  There are many other families of Artin-Schreier curves that may provide better parameters.  In addition, the examples in this paper have been restricted to codes defined over $\F_{p^2}$.  The number of points and arithmetic structure of these surfaces over larger fields may also yield interesting parameter ranges.  Further, the geometric intuition employed here would apply to any surface that can be intersected with planes and lines in a similar fashion.  For example, degree $m$ cyclic covers of $\mathbb{P}^2$ of the form $y^m=f(x,z)$, where $\gcd(m,p)=1$, may provide other interesting examples, with middle codes given by Kummer curves.

%\bibliographystyle{acm}
%\bibliography{HLRC-biblio}

\end{document}